\documentclass[12pt,reqno]{amsart}
\usepackage{amsthm,amssymb,latexsym,cite,amsmath,hyperref}
\usepackage{fullpage}

\newtheorem{theorem}{Theorem}[section]
\newtheorem{lemma}[theorem]{Lemma}

\newtheorem{definition}[theorem]{Definition}

\numberwithin{equation}{section}
\begin{document}

\title{The distribution of $r$-free numbers in arithmetic progressions}
\author{D. Jason Gibson}
\address{Eastern Kentucky University}
\email{jason.gibson@eku.edu}
\date{\today}

\begin{abstract}

Let $r\ge 2$.  A positive integer $n$ is called $r$-free
if $n$ is not divisible by the $r$-th power of a prime.
Generalizing earlier work of Orr, we provide
an upper bound of Bombieri-Vinogradov type
for the $r$-free numbers in arithmetic progressions.
\end{abstract} 

\maketitle
\section{Introduction}

Let $\mu_r(n)$ be the characteristic function of the $r$-free
numbers, so that, with $\mu(n)$ denoting the M\"obius
function, we have
\begin{equation}
\mu_r(n) = \sum_{d^r|n}\mu(d).
\end{equation}
To count the $r$-free numbers within arithmetic progressions,
consider the expression
\begin{equation}
R(x;k,l)=\sum_{\substack{n\le x\\n\equiv l\bmod k}}\mu_r(n)=
\sum_{\substack{n\le x\\n\equiv l\bmod k}}\sum_{d^r|n}\mu(d).
\end{equation}
Defining an arithmetic function $f=f_r$ simplifies
the statement of the main term in the asymptotic
count of these $r$-free numbers.
\begin{definition}[An arithmetic function]
For $r\ge 1$, define the arithmetic function $f=f_r$ by
\begin{equation}
f(k)=\prod_{p\nmid k }\left(1-\frac{1}{p^r}\right).
\end{equation}
\end{definition}

Suppose that $(l,k)=g$, where $g$ is $r$-free.  Write
$k=gs$ and $l=gt$.
Now, define the error term $E(x;k,l)$ by
\begin{equation}
E(x;k,l) = R(x;k,l) - \frac{x}{k}\frac{\phi(k)}{g\phi(s)}f(k).
\end{equation}

The survey of Pappalardi \cite{papp} provides an organized
overview of some of the main lines of work on $r$-free
numbers, and, in addition, states some open problems.

Among other results, the paper of Meng \cite{meng}
establishes a new upper bound of Barban-Davenport-Halberstam
type for $r$-free numbers in arithmetic progressions.
Jancevskis \cite{jan}, examining a variation of this
problem distinct from the approaches of Meng and Orr,
used sieve sequences to produce upper bounds
for squarefree numbers in arithmetic progressions.
For the case $r=2$, the squarefree numbers, in \cite{orr_phd} (also see\cite{orr_jnt}), Orr proved both a result of this type and the 
Bombieri-Vinogradov type result
\begin{theorem}
For any constant $A$,
\begin{equation}
\sum_{k\le \frac{x^{2/3}}{\log^{A+1}x}}\max_{(l,k)=\text{squarefree}} |E(x;k,l)|\ll \frac{x}
{\log^Ax}.
\end{equation}
\end{theorem}
In the sequel, we generalize this result to $r$-free numbers.

\section{Distribution, on average}

\begin{theorem}\label{rfree_t}
Let $r\ge 2$ be an integer.  Then, for any constant $A$,
\begin{equation}
\sum_{k\le \frac{x^{1-1/(r+1)}}{\log^{A+r-1}x}} \max_{(l,k)=r\text{-free}}|E(x;k,l)|\ll
\frac{x}{\log^Ax}.
\end{equation}
\end{theorem}

\begin{proof} We follow the argument of Orr \cite{orr_phd} closely.
A M\"obius sum detects $r$-free numbers. Changing
the order of summation, along with the estimation of some sums,
allows an error term to be split off from the main asymptotic term.
We obtain
\begin{align}
R(x;k,l)&=\sum_{\substack{m\le x/g\\
m\equiv t\bmod s\\
(m,g)=1}}\sum_{d^r|m}\mu(d)
=\sum_{\substack{d\le (x/g)^{1/r}\\
(d,g)=1}}\sum_{\substack{m\le x/g\\
m\equiv t\bmod s\\
(m,g)=1\\
m\equiv 0\bmod d^r}}\mu(d)\\
&=\sum_{\substack{d\le (x/g)^{1/r}\\
(d,g)=1\\
(d,s)=1}}\sum_{\substack{u\le x/(gd^r)\\
u\equiv t(d^r)^{-1}\bmod s\\
(u,g)=1}} 1\\
&=\sum_{\substack{
d\le z\\
(d,k)=1}}
\mu(d)
\sum_{\substack{
u\le x/(gd^r)\\
u\equiv t(d^r)^{-1}\bmod s\\
(u,g)=1}}
1
+
\sum_{\substack{
z<d\le (x/g)^{1/r}\\
(d,k)=1}}
\mu(d)
\sum_{\substack{
u\le x/(gd^r)\\
u\equiv t(d^r)^{-1}\bmod s\\
(u,g)=1}}
1.
\end{align}

Invoking Lemmas \ref{small_d} and \ref{large_d}, this yields
\begin{align}
R(x;k,l)=\frac{x}{k}\frac{\phi(k)}{g\phi(s)}f(k)+O\left(2^{\omega(g)}z
+r^{\omega(s)}\left(\frac{x}{kz^{r-1}}+\frac{x}{gz^r}\right)\right).
\end{align}
The choice $z=x^{1/(1+r)}$ gives
\begin{align}
E(x;k,l)&=R(x;k,l)-\frac{x}{k}\frac{\phi(k)}{g\phi(s)}f(k)\\
&\ll 2^{\omega(k)}x^{1/(r+1)}+
r^{\omega(k)}\left(
\frac{1}{k}x^{2/(r+1)}+x^{1/(r+1)}\right).
\end{align}

Finally, the use of Lemma \ref{divisors} gives the estimate
\begin{align}
\sum_{k\le \frac{x^{1-1/(r+1)}}{\log^{A+r-1}x}} \max_{(l,k)=r\text{-free}}|E(x;k,l)|&\ll
\sum_{k\le \frac{x^{1-1/(r+1)}}{\log^{A+r-1}x}} r^{\omega(k)}\left(
\frac{1}{k}x^{2/(r+1)}+x^{1/(r+1)}\right)\\
&\ll \sum_{k\le \frac{x^{1-1/(r+1)}}{\log^{A+r-1}x}} \tau_r(k)\left(
\frac{1}{k}x^{2/(r+1)}+x^{1/(r+1)}\right)\\
&\ll \frac{x}{\log^A x},
\end{align}
completing the proof of Theorem \ref{rfree_t}.
\end{proof}

\section{Some estimates}

\begin{lemma}[Small $d$ estimate]\label{small_d}
Let $r\ge 2$ be an integer.
Then
\begin{equation}
\sum_{\substack{
d\le z\\
(d,k)=1}}
\mu(d)
\sum_{\substack{
u\le x/(gd^r)\\
u\equiv t(d^r)^{-1}\bmod s\\
(u,g)=1}}
1
= 
\frac{x}{k}\frac{\phi(k)}{g\phi(s)}f(k)+O\left(\frac{x}{k}z^{1-r}+2^{\omega(g)}z\right).
\end{equation}
\end{lemma}

\begin{proof} We have
\begin{align}
\sum_{\substack{
d\le z\\
(d,k)=1}}
\mu(d)
\sum_{\substack{
u\le x/(gd^r)\\
u\equiv t(d^r)^{-1}\bmod s\\
(u,g)=1}}
1
&=
\sum_{\substack{d\le z\\
(d,k)=1}} \mu(d)\sum_{\substack{
u\le x/(gd^r)\\
u\equiv t(d^r)^{-1}\bmod s}} \sum_{\substack{
v|g\\
v|u}}\mu(v)\\
&=\sum_{\substack{v|g\\
(v,s)=1}}
\mu(v)
\sum_{\substack{
d\le z\\
(d,k)=1}}
\mu(d)
\sum_{\substack{
h\le x/(gd^rv)\\
h\equiv t(d^rv)^{-1}\bmod s}}
1\\
&=\sum_{\substack{v|g\\
(v,s)=1}}
\mu(v)
\sum_{\substack{
d\le z\\
(d,k)=1}}
\mu(d)\left(
\frac{x}{gd^rv}\frac{1}{s}+O(1)\right)\\
&=\frac{x}{k}\sum_{\substack{
v|g\\
(v,s)=1}}
\frac{\mu(v)}{v}
\sum_{\substack{
d\le z\\
(d,k)=1}}
\frac{\mu(d)}{d^r}+O\left(2^{\omega(g)}z\right).
\end{align}
A completion and rearrangement of that latter term
yields the conclusion, becoming
\begin{align}
\frac{x}{k}\frac{\phi(k)}{g\phi(s)}\left(
\sum_{(d,k)=1}\frac{\mu(d)}{d^r}+O\left(\sum_{d>z}\frac{1}{d^r}\right)\right)
&+O\left(2^{\omega(g)}z\right)\\
&=\frac{x}{k}\frac{\phi(k)}{g\phi(s)}f(k)+O\left(\frac{x}{k}z^{1-r}+
2^{\omega(g)}z\right).
\end{align}
\end{proof}

\begin{lemma}[Large $d$ estimate]\label{large_d}
Let $r\ge 2$ be an integer.
Then
\begin{equation}
\sum_{\substack{
z<d\le (x/g)^{1/r}\\
(d,k)=1}}
\mu(d)
\sum_{\substack{
u\le x/(gd^r)\\
u\equiv t(d^r)^{-1}\bmod s\\
(u,g)=1}}
1
\ll r^{\omega(s)}\left(\frac{x}{kz^{r-1}}+\frac{x}{gz^r}\right).
\end{equation}
\end{lemma}

\begin{proof}
We have
\begin{align}
\sum_{\substack{
z<d\le (x/g)^{1/r}\\
(d,k)=1}}
\mu(d)
\sum_{\substack{
u\le x/(gd^r)\\
u\equiv t(d^r)^{-1}\bmod s\\
(u,g)=1}}
1&\ll
\sum_{\substack{
z<d\le (x/g)^{1/r}\\
(d,k)=1}}
\sum_{\substack{
u\le x/(gd^r)\\
u\equiv t(d^r)^{-1}\bmod s\\}}
1\\
&\ll \sum_{u\le x/(gz^r)}\sum_{\substack{d\le (x/gu)^{1/r}\\
d^r\equiv tu^{-1}\bmod s}}1.
\end{align}

The number of solutions to $d^r\equiv tu^{-1}\bmod s$
can be bounded as $\ll r^{\omega(s)}$ (see, e.g., 4.2.1 and 4.2.2 of Ireland and Rosen \cite{IR}).  Our sum can then be estimated as
\begin{align}
\sum_{u\le x/(gz^r)}\sum_{\substack{d\le (x/gu)^{1/r}\\
d^r\equiv tu^{-1}\bmod s}}1&\ll
\sum_{u\le x/(gz^r)}r^{\omega(s)}\left(\frac{1}{s}\left(\frac{x}{gu}\right)
^{1/r}+1\right)\\
&\ll r^{\omega(s)}\left(\frac{1}{s}\left(\frac{x}{g}\right)^{1/r}\left(
\frac{x}{gz^r}\right)^{1-1/r}+\frac{x}{gz^r}\right)\\
&\ll r^{\omega(s)}\left(\frac{x}{kz^{r-1}}+\frac{x}{gz^r}\right).
\end{align}
\end{proof}

\begin{lemma}[Sum of divisors estimate] \label{divisors}
Let $r\ge 1$, and let $\tau_r(n)$ denote the number of representations of
$n$ as the product of $r$ natural numbers.  Then
\begin{equation}
\sum_{n\le x}\tau_r(n)\ll x(\log x)^{r-1}.
\end{equation}
\end{lemma}

\begin{proof}
This estimate can be found in Chapter 1.6 of Iwaniec and Kowalski
\cite{IK}.
\end{proof}

\section{Final remarks}

The approach taken here to this problem does not
significantly exploit much of the averaging or cancellation
which should occur, and it would be interesting to
remedy this defect.
In a survey \emph{The distribution of
sequences in arithmetic progressions} Halberstam \cite{hal}
wrote, \lq\lq \ldots although Orr's argument
takes little advantage of averaging, it does seem
hard to improve upon.\rq\rq 

\end{document}